\documentclass{amsart}[12pt]
\usepackage{xypic}
\usepackage{eucal}

\usepackage{amsmath}
\usepackage{amstext}
\usepackage{amssymb}
\usepackage{amsthm}
\usepackage{amscd}

\usepackage{lscape}
\usepackage{longtable}

\usepackage{epsfig}
\input txdtools
\let\et=\etexdraw
\def\etexdraw{\drawbb\et}

\newtheorem{teo}{Theorem}

\newtheorem{lemma}[teo]{Lemma}
\newtheorem{prop}[teo]{Proposition}
\newtheorem{corol}[teo]{Corollary}

\theoremstyle{definition}
\newtheorem{defi}[teo]{Definition}

\newtheorem{esem}[teo]{Example}

\theoremstyle{remark}
\newtheorem{oss}[teo]{Remark}

\DeclareMathOperator{\Ker}{Ker}
\DeclareMathOperator{\Coker}{Coker}
\DeclareMathOperator{\Imm}{Im}
\DeclareMathOperator{\HH}{H}

\DeclareMathOperator{\Ann}{Ann}
\DeclareMathOperator{\Spec}{Spec}
\DeclareMathOperator{\Supp}{Supp}
\DeclareMathOperator{\Nil}{Nil}

\DeclareMathOperator{\Hom}{Hom}
\DeclareMathOperator{\Ext}{Ext}

\DeclareMathOperator{\E}{E}
\DeclareMathOperator{\di}{d}
\DeclareMathOperator{\Mat}{Mat_{s,t}(A)}
\DeclareMathOperator{\HSL}{HSL}
\DeclareMathOperator{\Ass}{Ass}

\newcommand{\W}{W^{-1}}

\newcommand{\ER}{\E_R(\mathbb{K})}

\newcommand{\Z}{\mathbb{Z}}

\newcommand{\U}{\mathcal{U}}
\newcommand{\K}{\mathbb{K}}

\newcommand{\m}{\mathfrak{m}}
\newcommand{\p}{\mathfrak{p}}

\newcommand{\Hs} {\HH^{\di}_{\mathfrak{m} S}(S)}
\newcommand{\Hw}{\HH^{\di}_{\mathfrak{m} S}(\bar\omega)}
\newcommand{\Hsw}{\HH^{\di-1}_{\mathfrak{m}S}(S/\bar\omega)}
\newcommand{\Hrw}{\HH^{\di-1}_{\mathfrak{m}}(R/\omega)}
\newcommand{\Hir} {\HH^i_{\mathfrak{m} }(R)}
\newcommand{\Hdr} {\HH^{\di}_{\mathfrak{m} }(R)}

\begin{document}

\title
[On the Upper Semi-continuity of HSL Numbers]
{On the upper semi-continuity of HSL numbers}

\author{Serena Murru}
\address{Department of Pure Mathematics,
University of Sheffield, Hicks Building, Sheffield S3 7RH, United Kingdom}
\email{pmp11sm@sheffield.ac.uk}


\begin{abstract}
Let $B$ be an affine Cohen-Macaulay algebra over a field of characteristic $p$.
For every prime ideal $\p\subset B$, let $H_\p$ denote $\HH^{\dim B_\p}_{\p B_\p}\left( \widehat{B_\p} \right)$.
Each such $H_\p$ is an Artinian module endowed with a natural Frobenius map $\Theta$ and if $\Nil(H_\p)$ denotes the set of all elements
in $H_\p$ killed by some power of $\Theta$ then a theorem by Hartshorne-Speiser and Lyubeznik shows that
there exists an $e\geq 0$ such that $\Theta^e \Nil(H_\p)=0$. The smallest such $e$ is the HSL-number of $H_\p$ which we denote $\HSL(H_\p)$.\\
The main theorem in this paper shows that for all $e>0$, the sets
$\{ \p\in\Spec B \,|\, \HSL(H_\p) < e \}$ are Zariski open, hence $\HSL$ is upper semi-continuous. An application of this result gives a global test exponent for the calculation of Frobenius closures of parameter ideals in Cohen-Macaulay rings. 
\end{abstract}

\maketitle

\section{Introduction}\label{section:introduction}
\emph{Throughout this paper every ring is assumed to be Noetherian, commutative, associative, with identity, and of prime characteristic $p$.}\\

Let $R$ be a ring and for every positive integer $e$ define the $e^{th}$-iterated Frobenius endomorphism $T^e\colon R\to R$  
to be the map $r\mapsto r^{p^e}$. For $e=1$, $R\to R$ is the natural Frobenius map on $R$. For any $R$-module $M$ we define $F_*^e M$ to be the Abelian group $M$ with $R$-module structure given by
$r \cdot m = T^e(r) m = r^{p^e} m$ for all $r\in R$ and $m\in M$. \\
We can extend this construction to obtain the Frobenius functor $F^e_R$ from $R$-modules to $R$-modules as follows. 
For any $R$-module $M$, we consider the $F_*^e R$-module $F_*^e R \otimes_R M$ and after identifying the rings $R$ and $F_*R$,
we may regard $F_*^e R \otimes_R M$ as an $R$-module and denote it $F^e_R(M)$ or just $F^e(M)$ when $R$ is understood. 
The functor $F^e_R(-)$ is exact when $R$ is regular, cf. \cite[Corollary 8.2.8]{2}, and for any matrix $C$ with entries in $R$, 
$F^e_R(\Coker C)$ is the cokernel of the matrix $C^{[p^e]}$ obtained from $C$ by raising its entries to the $p^e$th power, cf. \cite{6}.
For any  $R$-module $M$ an additive map 
$\varphi: M \rightarrow M$ is an \emph{$e^{th}$-Frobenius map} 
if it satisfies $\varphi(r m)=r^{p^e} \varphi(m)$ for all $r\in R$ and $m\in M$. 
Note that there is a bijective correspondence between $\Hom_R(M,F_*M)$ and the Frobenius maps on $M$.

For every $e\geq 0$ let $\mathcal{F}^e(M)$ be the set of all Frobenius maps on $M$. Each $\mathcal{F}^e(M)$ is an $R$-module: for all $\varphi\in \mathcal{F}^e(M)$ and $r\in R$ the map $r\varphi$ defined as $(r\varphi)(m)=r\varphi(m)$ is in $\mathcal{F}^e(M)$ for all $m\in M$.
If $\varphi\in\mathcal{F}^e(M)$ we can define for $i\geq 0$
the $R$-submodules $M_i=\{ m\in M \,|\, \varphi^i (m) = 0\}$.
We define the \emph{submodule of nilpotent elements in $M$} as
$$\Nil(M)=\cup_{i\geq 0} M_i.$$

\begin{teo}[cf.~Proposition 1.11 in \cite{5} and Proposition 4.4 in \cite{9}] \label{incl}
If $(R,\m)$ is a complete regular ring, $M$ is an Artinian $R$-module  and $\varphi\in\mathcal{F}^e(M)$
then the ascending sequence $\{ M_i \}_{i\geq 0}$ above stabilises, i.e.,
there exists an $e\geq 0$ such that $\varphi^e (\Nil(M))=0$.
\end{teo}

\begin{defi}\label{defi:nilp}
We define the \emph{HSL number} or \emph{index of nilpotency} of $\varphi$ on $M$, denoted $\HSL(M)$, 
to be the smallest integer $e$ at which $\varphi^e (\Nil(M))=0$, or $\infty$ if no such $e$ exists.
\end{defi}
We can rephrase Theorem \ref{incl} by saying that under the hypothesis of the theorem, $\HSL(M)<\infty$.
\\

Another way of describing a Frobenius map $\varphi\colon M\to M$ on an $R$-module $M$ is to think of $M$ as a module over a certain skew-commutative ring $R[\theta;f^e]$ where the latter is defined as follows. $R[\theta;f^e]$ is the free $R$-module $\bigoplus_{i=0}^\infty R\theta^i$ endowed with the further non-commutative operation $\theta s=s^{p^e}\theta$ for every $s\in S$. Therefore it is equivalent to say that $M$ is an $R$-module with a Frobenius action given by $\varphi$ and that $M$ is an $R[\theta;f^e]$-module with module structure given by $\theta m=\varphi(m)$. 
\\

The action of Frobenius on a local cohomology module is constructed as follows.
Any $R$-linear map $M\to N$ induces a map $\HH^i_I(M)\to\HH^i_I(N)$ for every $i$.
The map $R\to F_*R$ sending $r\mapsto F_*r^p$ is $R$-linear because $F_*r^p=r\cdot F_* 1$ and so it induces for every $i$ a map $\HH^i_I(R)\to \HH^i_I(F_*R)=\HH^i_{IF_*R}(F_*R)=\HH^i_{F_*I^{[p^e]}}(F_*R)=\HH^i_{F_*I}(F_*R)=F_*\HH^i_I(R)$ where in the first equality we used the Independence Theorem for local cohomology \cite[Proposition 4.1]{3} and in the third that the ideals $I$ and $I^{[p^e]}$ have same radical cf. \cite[Proposition 3.1.1]{3}. So we get an $R$-linear map $\HH^i_I(R)\to F_*\HH^i_I(R)$ which is the same as a Frobenius map $\HH^i_I(R)\to \HH^i_I(R)$.



If $(R,\m)$ is regular of dimension $d$ and $x_1,\cdots,x_d$ is a system of parameters for $R$ then we can write $\Hdr$ as the direct limit 
$$\frac{R}{(x_1,\cdots,x_d)R}\overset{x_1\cdots x_d}{\longrightarrow}\frac{R}{(x_1^2,\cdots,x_d^2)R}\overset{x_1\cdots x_d}{\longrightarrow}\cdots$$
where the maps are the multiplication by $x_1\cdots x_d$.\\
Another way to describe the natural Frobenius action on $\Hdr$ is the following. The natural Frobenius map on $R$ induces a natural Frobenius map on $\Hdr$ in the following way; a map $\phi\in \mathcal{F}^e(\Hdr)$ is defined on the direct limit above by mapping the coset $a+(x_1^t,\cdots,x_d^t)R$ in the t-th component to the coset $a^{p^e}+(x_1^{tp^e},\cdots,x_d^{tp^e})R$ in the $tp^e$-th component.

\begin{defi} 
A local ring $(R,\m)$ is $F$-\emph{injective} if the natural Frobenius map $\Hir\to \Hir$ is injective for all $i$.
\end{defi}

The structure of this paper is the following; in Section \ref{The $I_e(-)$ operator} we define the operator $I_e(-)$ and in the case of a 
polynomial ring $A$ we show that it commutes with completions and localisations with respect to any multiplicatively closed subset of $A$. In section \ref{sect:deltapsi} we define the $\Delta^e$- and the $\Psi^e$-functors, cf. \cite{7}. 
In Section \ref{sect:loc_case} we consider a quotient $S$ of a regular local ring $(R,\m)$ and we give an explicit description of the $R$-module 
$\mathcal{F}^e(\Hs)$ consisting of all  $e^{th}$-Frobenius maps acting on the top local cohomology module $\Hs$. We then give a formula to compute $\HSL(\Hs)$ when $S$ is a Cohen-Macaulay domain. 
In Section \ref{sect:non-loc_case} we prove that the set 
$\mathcal{B}_e= \left\{\p\in \Spec(A)\vert \HSL \left(\HH^{d_{(\p)}}_{\p \widehat{B}_\p}(\widehat{B}_\p)\right)<e \right\}$, where $B$ is a 
quotient of a polynomial ring $A$, is a Zariski open set. Note that this result generalises the openness of the F-injective locus. 
Furthermore, in Section \ref{section:1stAlgorithm} we provide an algorithm for computing the $\HSL$-loci and we give an example. The results of Section \ref{Section:application} give a global test exponent for the calculation of Frobenius closures of parameter ideals in Cohen-Macaulay rings.

\section{The $I_e(-)$ operator}\label{The $I_e(-)$ operator}
In this section we define the operator $I_e(-)$ which has been introduced in \cite{6}, and in \cite{3} with the notation $(-)^{[1/p^e]}$. We will show that this commutes with localisations and completions.

For any ideal $I$ of a ring $R$, we shall denote by $I^{[p^e]}$ the $e^{th}$-Frobenius power of $I$, i.e. the ideal generated by 
$\{a^{p^e}\vert a\in I\}$.

\begin{defi}\label{def_I_e}
If $R$ is a ring and $J\subseteq R$ an ideal of $R$ we define $I_e(J)$ to be the smallest ideal $L$ of $R$ such that its $e^\text{th}$-Frobenius power $L^{[p^e]}$ contains $J$.
\end{defi}

In general, such an ideal may not exist; however it does exist in polynomial rings and power series rings, cf \cite[Proposition 5.3]{7}.\\

Let $A$ be a polynomial ring $\K[x_1,\dots,x_n]$ and $W$ be a multiplicatively closed
subset of $A$ and $J\subset A$ an ideal.


\begin{lemma}\label{prem2_I_e_commutes}
If $L\subseteq \W A$ is any ideal then $L^{[p^e]}\cap A=(L\cap A)^{[p^e]}$.
\end{lemma}

\begin{proof}
Let $\frac{g_1}{1}, \dots, \frac{g_s}{1}$ be a set of generators for $L$ and let $G$ be the ideal of $A$ generated by $g_1, \dots, g_s$. Then we can write $L^{[p^e]}\cap A=\sum_{w\in W} (G^{[p^e]} :_A w)$ and 
$(L\cap A)^{[p^e]}=\sum_{w\in W} (G :_A w)^{[p^e]}$. 
Since $A$ is regular, for any $w\in W$, $w^{p^e}$ is in $W$ and  $(G^{[p^e]} :_A w^{p^e})= (G:_A w)^{[p^e]}$ so $(L\cap A)^{[p^e]}\subseteq L^{[p^e]}\cap A$.
Also  $ (G^{[p^e]} :_A w) \subseteq  (G^{[p^e]} :_A w^{p^e}) $ so $L^{[p^e]}\cap A\subseteq(L\cap A)^{[p^e]}$.
\end{proof}

\begin{lemma}\label{prel_thm_I_eComm}
If $J$ is any ideal of $A$ then $I_e(\W J)$ exists for any integer $e$ and equals $W^{-1}I_e(W^{-1}J\cap A)$.
\end{lemma}
\begin{proof}
Let $L\subseteq \W A$ be an ideal such that $\W J\subseteq L^{[p^e]}$, then $\W I_e(\W J\cap A)\subseteq L$; in fact, $\W J\cap A\subseteq L^{[p^e]}\cap A=(L\cap A)^{[p^e]}$ where the equality follows from Lemma \ref{prem2_I_e_commutes}. Thus $I_e(\W J\cap A)\subseteq L\cap A$ so $\W I_e(\W J\cap A)\subseteq \W(L\cap A) \subseteq L$. Hence $\W I_e(\W J\cap A)$ is contained in all the ideals $L$ such that $\W J\subseteq L^{[p^e]}$. If we show that $\W J\subseteq(\W I_e(\W J\cap A))^{[p^e]}$ then $I_e(\W J)$ exists and equals $W^{-1}I_e(W^{-1}J\cap A)$. But since $\W J\cap A\subseteq I_e(\W J\cap A)^{[p^e]}$ then using Lemma \ref{prem2_I_e_commutes} we obtain $\W J= \W (\W J\cap A)\subseteq \W (I_e(\W J\cap A)^{[p^e]})=(\W I_e(\W J\cap A))^{[p^e]}$.
\end{proof}

\begin{prop} \label{I_eCommutes}
Let $\widehat{A}$ denote the completion of $A$ with respect to any prime ideal and $W$ any multiplicatively closed
subset of $A$. Then the following hold:

\begin{enumerate}
\item $I_e(J\otimes_A \widehat{A})=I_e(J)\otimes_A\widehat{A}$,  for any ideal $J\subseteq A$;
\item  $W^{-1}I_e(J)=I_e(W^{-1}J).$
\end{enumerate}
\end{prop}

\begin{proof}

\begin{enumerate}
\item Write $\widehat{J}=J\otimes_A \widehat{A}$. Since $I_e(\widehat{J})^{[p^e]}\supseteq\widehat{J}$ using \cite[Lemma 6.6]{11} we obtain
$$( I_e(\widehat{J})\cap A)^{[p^e]}=I_e(\widehat{J})^{[p^e]}\cap A\supseteq  \widehat{J}\cap A = J.$$
 But $I_e(J)$ is the smallest ideal such that $I_e(J)^{[p^e]}\supseteq J$, so $I_e(\widehat{J})\cap A \supseteq I_e(J)$ and hence
$I_e(\widehat{J})\supseteq(I_e( \widehat{J})\cap A)\otimes_A\widehat{A}\supseteq I_e(J)\otimes_A\widehat{A}$.\\
On the other hand, $(I_e(J)\otimes_A\widehat{A})^{[p^e]}=I_e(J)^{[p^e]}\otimes_A\widehat{A}\supseteq J\otimes_A \widehat{A}$ and so $I_e(J\otimes_A \widehat{A})\subseteq I_e(J)\otimes_A\widehat{A}$.

\item Since $J\subseteq W^{-1}J\cap A$, $I_e(J)\subseteq I_e(W^{-1}J\cap A)$, and so $W^{-1}I_e(J)\subseteq W^{-1}I_e(W^{-1}J\cap A)$. By  Lemma \ref{prel_thm_I_eComm}, $W^{-1}I_e(W^{-1}J\cap R)=I_e(W^{-1}J)$ hence $W^{-1}I_e(J)\subseteq I_e(W^{-1}J)$.\\
For the reverse inclusion it is enough to show that $W^{-1}J\subseteq (W^{-1}I_e(J))^{[p^e]}$ because from this it follows that  $ I_e(W^{-1}J)\subseteq W^{-1}I_e(J)$ which is what we require. Since $J\subseteq I_e(J)^{[p^e]}$ then $W^{-1}J\subseteq W^{-1}( I_e(J)^{[p^e]})=(W^{-1} I_e(J))^{[p^e]}$ where in the latter equality we have used Lemma \ref{prem2_I_e_commutes}.
\end{enumerate}
\end{proof}


\section{The $\Delta^e$- and $\Psi^e$-functors} \label{sect:deltapsi}

Let $(R,\m)$ be a complete and local ring and let $(-)^\vee$ denote the \emph{Matlis dual}, i.e. the functor $\Hom_R(-,E_R)$, where 
$E_R=\ER$ is the injective hull of the residue field $\K$ of $R$. In this section we recall the notions of $\Delta^e$-functor and $\Psi^e$-functor which have been described in more detail in \cite[Section 3]{7}.\\
Let $\mathcal{C}^e$ be the category of Artinian $R[\theta,f]$-modules and $\mathcal{D}^e$ the category of $R$-linear maps
$\alpha_M\colon M\to F^e_R(M)$ with $M$ a Noetherian $R$-module and where a morphism between
$M\overset{\alpha_M}{\to}  F^e_R(M)$ and $N\overset{\alpha_N}{\to} F^e_R(N)$ is a commutative diagram of $R$-linear maps:
\[
 \xymatrix{
M \ar[r]^{h}\ar[d]_{\alpha_M} & N  \ar[d]_{\alpha_N}\\
 F^e_R(M) \ar[r]^{ F^e_R(h)}  &  F^e_R(N) .\\
}
\]
We define a functor $\Delta^e : \mathcal{C}^e \rightarrow \mathcal{D}^e$ as follows: given an  $e^{th}$-Frobenius map $\theta$ of the Artinian $R$-module $M$, we obtain an $R$-linear map
$\phi: F_*^e(R) \otimes_R M \rightarrow M$ which sends $F_*^e r \otimes m$ to $r \theta m$. 
Taking Matlis duals, we obtain the $R$-linear map
$$ M^\vee \rightarrow (F_*^e(R) \otimes_R M)^\vee \cong F_*^e(R) \otimes_R M^\vee$$
where the last isomorphism is the functorial isomorphism described in \cite[Lemma 4.1]{9}.
This construction can be reversed, yielding a functor  $\Psi^e : \mathcal{D}^e \rightarrow \mathcal{C}^e$ such that
$\Psi^e \circ \Delta^e $ and $\Delta^e\circ \Psi^e$ can naturally be identified with the identity functor.
See  \cite[Section 3]{7} for the details of this construction.

\section{The local case}\label{sect:loc_case}
In this section we give an explicit formula for the $\HSL$-numbers. \\

Let $(R,\m)$ be a complete, regular and local ring, $I$ an ideal of $R$ and write $S=R/I$.
Let $d$ be the dimension of $S$ and suppose $S$ is Cohen-Macaulay with canonical module $\bar\omega$. Assume that $S$ is generically Gorenstein (i.e. each localisation of $S$ at a prime ideal is Gorenstein) so that 
$\bar\omega\subseteq S$ is an ideal of $S$, (see. \cite[Proposition 2.4]{12}), and consider the following short exact sequence:
$$0\to \bar\omega\to S \to S/\bar\omega\to 0$$
that induces the long exact sequence
$$\cdots\to \HH^{\di-1}_{\m S}(S) \to \HH^{\di-1}_{\m S}(S/\bar\omega)\to \HH^{\di}_{\m S}(\bar\omega)\to \HH^{\di}_{\m S}(S)\to 0.$$
Since $S$ is Cohen-Macaulay, the above reduces to
\begin{equation}\label{ex_seq_loc_coho}
0\to \Hsw \to \Hw \to \Hs\to 0.
\end{equation}

As noted in the introduction, a natural Frobenius map acting on $S$ induces a natural Frobenius map acting on $\Hs$. 
The following theorem gives a description of the
natural Frobenius (up to a unit) which we will later use in Theorem \ref{teo:fraction}.


\begin{teo}[cf.~ in Example~3.7 \cite{11}]   \label{teo:unique_generator}
Let $\mathcal{F}^e:=\mathcal{F}_e(\Hs)$ be the $R$-module consting of all  $e^{th}$-Frobenius maps acting on $\Hs$. Then $\mathcal{F}^e$ is generated by one element which corresponds, up to unit, to the natural Frobenius map.
\end{teo}
We aim to give an explicit description of the $R$-module $\mathcal{F}^e$ and consequently of the natural Frobenius map that generates it.

\begin{oss}
The inclusion $\bar\omega\to S$ is $R[\theta,f^e]$-linear where $\theta s=s^{p^e}$ acts on $\bar{\omega}$ by restriction. This induces an $R[\theta;f^e]$-linear map $\Hw\overset{\alpha}{\to}\Hs\to 0$ where the structure of $R[\theta,f^e]$-module on $\Hw$ is obtained from the one on $\bar\omega\subseteq S$ and where $\Hs$ has a natural structure of $R[\theta; f^e]$-module as we have seen in the introduction.   

Since any kernel of an $R[\theta; f^e]$-map is an $R[\theta; f^e]$-module, $\ker(\alpha)= \Hsw $ is an $R[\theta; f^e]$-module as well. Hence the sequence
$0\to \Hsw \to \Hw \to \Hs\to 0$
is an exact sequence of $R[\theta; f^e]$-modules.
\end{oss}

Identifying $\Hw$ with $E_S=\Ann_{E_R}(I)$ it is clear that $\Hsw$ must be of the form $\Ann_{E_S}(J)$ for a certain ideal $J\subseteq R$. More precisely we have the following:
\begin{lemma}\label{annomega}
$\Hsw$ and $\Ann_{E_S}(\bar\omega)$ are isomorphic.
\end{lemma}
\begin{proof}
If $\omega$ is the preimage of $\bar\omega$ in $R$ then $\Hsw\cong\Hrw$. Since $\Hrw$ is an $R$-submodule of $E_S=\Ann_{E_R}I\subseteq E_R$ then 
 $\Hrw=\Ann_{E_R}\left(0\colon_R \Hrw\right)$.
The fact that $(0:\Hrw)=\omega$ follows from \cite[Theorem 2.17]{10} replacing $R$ with $R/\omega$ and $(0)$ with $\omega$ and using the fact that $\omega$ is unmixed.
\end{proof}

\begin{oss}\label{u_nu}
All $R[\theta;f^e]$-module structures on $\Ann_{E_R}(I)=E_S$ have the form $uF$ where $F$ is the natural Frobenius map on $E_R$ and $u\in \left(I^{[p^e]}:I\right)$. The identification $\Hw$ with $E_S$ endows $E_S$ with a Frobenius map which then has to be of the form $uF$ with $u\in \left(I^{[p^e]}:I\right)$. \\
In general if we start with an $R[\theta;f]$-module $M$, we can consider $M$ as an $R[\theta_e,f^e]$-module where $f^e\colon R\to R$, $f^e(a)=a^{p^e}$, $\theta_e(m)=\theta^e(m)$. In our case, for $M=E_S$ the action of $\theta_e$ on $E_S$ is:
$$\theta_e=\underbrace{\theta\circ \cdots\circ \theta}_{e \text{ times}}=(uF)^e=u^{\nu_e}F^e$$
where $\nu_e=1+p+\cdots+p^{e-1}$ when $e>0$ and $\nu_0=0$.
Therefore when we apply the $\Delta^e$-functor to $E_S\in \mathcal{C}^e$ we obtain the map
$$R/I\overset{u^{\nu_e}}{\to} R/I^{[p^e]}.$$
\end{oss}

\begin{teo}\label{teo:fraction}
The $R$-module consting of all $e^{th}$-Frobenius maps acting on $\Hs$ is of the form
$$\mathcal{F}^e=\frac{\left(I^{[p^e]}:I\right)\cap\left({\omega}^{[p^e]}:{\omega}\right)}{I^{[p^e]}}$$
where ${\omega}$ is the preimage of $\bar\omega$ in $R$.
\end{teo}

\begin{proof}
By Lemma \ref{annomega} we can rewrite (\ref{ex_seq_loc_coho}) as
\begin{equation}\label{ex_seq_loc_ann}
0\to \Ann_{E_S}(\bar\omega) \to \Ann_{E_R}(I) \to \Hs\to 0.
\end{equation}
Apply the $\Delta^e$-functor to the latter short exact sequence. When we apply it to $E_S=\Ann_{E_R}(I)$ and $\Ann_{E_S}(\omega)$ we obtain respectively $\Delta^e(\E_S)=R/I\overset{u^{\nu_e}} {\to}R/I^{[p^e]}$ and $\Delta^e(\Ann_{E_S}(\omega))=R/\omega\to R/\omega^{[p^e]}$. Thus the inclusion $\Ann_{E_S}(\omega)\to E_S$ yelds to the diagram
\[
 \xymatrix{
 0\ar[r]& (\Hs)^\vee \ar[r] \ar[d]& R/I \ar[r]\ar[d]_{u^{\nu_e}} & R/\omega\ar[r]\ar[d] &0\\
 0\ar[r] &F^e_R(\Hs)^\vee \ar[r]& R/I^{[p^e]} \ar[r] & R/\omega^{[p^e]} \ar[r] &0.
}
\]
Now, we can identify $(\Hs)^\vee$ with $\omega/I$ and $F^e_R(\Hs)^\vee$ with $\omega^{[p^e]}/I^{[p^e]}$.
Therefore when we apply $\Delta^e$ to the sequence (\ref{ex_seq_loc_ann}) we obtain the short exact sequence in $\mathcal{D}^e$: 
\[
 \xymatrix{
0 \ar[r] & \bar\omega/I\ar[r]\ar[d] & R/I \ar[r]\ar[d]_{u^{\nu_e}} & R/\omega\ar[r]\ar[d] &0\\
0 \ar[r] & \bar\omega^{[p^e]}/I^{[p^e]}\ar[r] & R/I^{[p^e]} \ar[r] & R/\omega^{[p^e]} \ar[r] &0
}
\]
where the central vertical map is the multiplication by $u^{\nu_e}$. The only way to make the diagram above commutative is that the other two vertical maps are also the multiplication by $u^{\nu_e}$.
It follows that $u\in \left(I^{[p^e]}\colon I\right)\cap\left(\omega^{[p^e]}\colon \omega\right)$. Finally consider the surjection
$$\varphi\colon\left(I^{[p^e]}\colon I\right)\cap\left(\omega^{[p^e]}\colon \omega\right)\to \mathcal{F}^e (\Hs)$$ then $u \in \Ker\varphi$ if and only if $u\colon \frac{\bar \omega}{I}\to \frac{\bar\omega^{[p^e]}}{I^{[p^e]}}$ is the zero map which happens if and only if $u\bar\omega\subset I^{[p^e]}\subseteq I$ i.e. $u\omega=0$. $\bar\omega$ contains a non-zero-divisor and since $\bigcup\Ass(I)=\bigcup\Ass\left(I^{[p^e]}\right)$ then $\bar\omega$ contains a non-zero-divisor modulo $I^{[p^e]}$, say $x$. So $ux\in I^{[p^e]}$ implies $u\in I^{[p^e]}$. Therefore $\Ker\varphi=I^{[p^e]}$.
\end{proof}

We prove now the main result of this section:

\begin{teo}\label{index_Hs}
$\HSL(\Hs)$ is the smallest integer $e$ for which
$$\frac{I_{e}(u^{\nu_e}\omega)}{I_{e+1}(u^{\nu_{e+1}}\omega)}=0$$
where $\omega$ is the preimage of $\bar\omega$ in $R$ and $\nu_e=1+p+\cdots+p^{e-1}$ when $e>0$ and $\nu_0=0$.
\end{teo}
\begin{proof}


For all $e\ge 0$ define $M_e=\left\{x\in \Hs\vert \theta^ex=0\right\}$ and note that $\{M_e\}_{e\geq 0}$ form an ascending sequence of $R[\theta,f^e]$-submodules of $\Hs$ that stabilises by Theorem \ref{incl}. 
Consider the short exact  sequence of $R[\theta,f^e]$-modules
$0\to \Hsw\to\Hw\to\Hs\to 0$
where the action of $\theta$ on $E_S=\Hw$ is given by $u^{\nu_e}F$ where $F$ is the natural Frobenius on $E_R$.
We have seen we can write this sequence as 
$$0\to \Ann_{E_S}(\bar\omega) \to \Ann_{E_R}(I) \to \Hs\to 0.$$
It follows that
$$\Hs\cong \frac{\Ann_{E_R}(I)}{ \Ann_{E_S}(\bar\omega)}.$$
Since each $M_e$ is a submodule of $\Hs$ then it is of the form $\frac{\Ann_{E_S} (L_e)}{\Ann_{E_S}(\bar\omega)}$ for some ideals $L_e\subseteq R$ contained in $I$. 
Apply the $\Delta^e$-functor to the inclusion $M_e\hookrightarrow \Hs$ to obtain

\[
 \xymatrix{
\omega/I \ar@{->>}[r]\ar[d]_{u^{\nu_e}} & \omega/L_e  \ar[d]\\
\omega^{[p^e]}/I^{[p^e]}\ar@{->>}[r]  & \omega^{[p^e]}/L_e^{[p^e]}\\
}
\]
where the map $\omega/I\to \omega^{[p^e]}/I^{[p^e]}$ is the multiplication by $u^{\nu_e}$ by Remark \ref{u_nu}.
It follows that the map $\omega/L_e\to \omega^{[p^e]}/L_e^{[p^e]}$ must be the multiplication by $u^{\nu_e}$ because of the surjectivity of the horizontal maps;  note that such a map is well defined because $u^{\nu_e}\omega\subseteq \omega^{[p^e]}$, and then $L_e\subseteq \omega$. Moreover $\omega/L_e\to \omega^{[p^e]}/L_e^{[p^e]}$ must be a zero-map by construction of $\Delta^e$.  Hence, $u^{\nu_e}\omega\subseteq L_e^{[p^e]}$ and for every $L_e$ with $u^{\nu_e}\omega\subseteq L_e^{[p^e]}$ the action of $\theta$ on $M_e$ is zero. We want the largest $M_e\subset \Hs$ for which $\theta^e$ acts as zero. The largest module $\frac{\Ann_{E_S}(L_e)}{\Ann_{E_S}(\bar{\omega})}$ killed by $\theta^e$ corresponds to the smallest $L_e$ such that $u^{\nu_e}\omega\subseteq L_e^{[p^e]}$ i.e. $L_e=I_e(u^{\nu_e}\omega)$.
\end{proof}

\begin{corol}\label{S_Finj_iff}
$S$ is $F$-injective if and only if $\omega=I_1(u\omega)$.
\end{corol}
\begin{proof}
$S$ is $F$-injective if and only if the index of nilpotency is zero i.e. if and only if $\omega=I_1(u\omega)$.
\end{proof}


\section{The non-local case}\label{sect:non-loc_case}

Let $A$ be a polynomial ring $\K[x_1,\cdots,x_n]$ with coefficients in a perfect field of positive characteristic $p$ and let $M$ be a finitely generated $A$-module generated by $g_1,\cdots,g_s$.
Let $e_1,\cdots, e_s$ be the canonical basis for $A^s$ and define the map
 \[
\xymatrix @R=0.2pt
{
A^s \ar[r]^{\varphi}  & M\\
e_i \ar@{|->}[r] & g_i.
}
\]
$\varphi$ is surjective and extends naturally to an $A$-linear map $J\colon A^t\to A^s$ with $\ker \varphi=\Imm J$.   Let $J_i$ be the matrix obtained from $J\in \Mat$ by erasing the $i^{th}$-row. With this notation we have the following:
\begin{lemma}\label{prel_G_i}
 $M$ is generated by $g_i$ if and only if $\Imm J_i=A^{s-1}$.
\end{lemma}
\begin{proof}
Firstly suppose $\Imm J_i=A^{s-1}$. We can add to $J$, columns of $\Imm J$ without changing its image so we can assume that $J$ contains the elementary vectors $e_1,\cdots,e_{i-1},e_{i+1},\cdots,e_s$:\[J= \left( \begin{array}{ccccccccccccccc}
a_{1,1} & a_{1,2} & \cdots & a_{1,n} & 1     & 0          &  \cdots & 0     \\
a_{2,1} & a_{2,2} & \cdots & a_{2,n} & 0     & 1          &  \cdots & 0     \\
 \vdots  & \vdots  &  & \vdots      &    \vdots  & \vdots  &\vdots         & \vdots    \\
  a_{i,1} & a_{i,2} & \cdots & a_{i,n} & b_1 & b_2      & \cdots & b_s\\
  \vdots  & \vdots  &  & \vdots      &    \vdots  &\vdots  &\vdots          & \vdots  \\
a_{m-1,1} & a_{m-1,2} & \cdots & a_{m-1,n} & 0     &  \cdots         &  1 & 0     \\
  a_{m,1} & a_{m,2} & \cdots & a_{m,n} & 0   & \cdots   & 0         & 1
\end{array} \right).\] 
where $a_{k,l}, b_j\in A$. In this way for every $j\neq i$ we have $g_j-b_j g_i=0$ i.e. $g_i$ generates $M$. Viceversa if $M$ is generated by $g_i$  then for all $j\neq i$ we can write $g_j=r_jg_i$ i.e. $g_j-r_jg_i=0$ and the relation $g_j-r_jg_i$ gives a relation  $e_j-r_je_i$ in the image of J, so $e_j-r_je_i\in \Ker\varphi=\Imm J$. Hence we can assume that $J$ contains a column whose entries are all zeros but in the $i$-th and $j$-th positions where there is $1$ and $r_j$ respectively. Consequently $J_i$ contains the $(s-1)\times (s-1)$ identity matrix.
\end{proof}

Let $W$ be a multiplicatively closed subset of $A$.
Localise the exact sequence $A^t\to A^s\to M\to 0$ with respect to $W$ obtaining the exact sequence $\W A^t\to \W A^s\to \W M\to 0$. With this notation we have:
\begin{prop}\label{generaters}
 $\W M$ is generated by $\frac{g_i}{1}$ if and only if $\W J_i= (\W A)^{s-1}$
\end{prop}
\begin{proof}
Apply Lemma \ref{prel_G_i} to the localised sequence $\W A^t\to \W A^s\to \W M\to 0$.
\end{proof}

Proposition \ref{generaters} is equivalent to saying that the intersection of $W$ with the ideal of $(s-1) \times (s-1)$ minors of $J_i$ is not trivial. So we have the following;

\begin{corol} \label{G_izariski}
Let $M$ be a finitely generated $A$-module and let  $g_1,\cdots,g_s$ be a set of generators for $M$. If $M$ is locally principal then for each $i=1,\cdots,s$
$$\mathcal{G}_i=\{\p\in\Spec(A) \ \vert M\widehat{A}_\p \text{\emph{ is generated by the image of} } g_i\}$$
is a Zariski open set and $\cup_i\mathcal{G}_i=\Spec(A)$. Moreover, $\mathcal{G}_i=V (J_i)^c$ for every $i=1,\cdots,s$. 
\end{corol}
\begin{proof}
$\p\in \mathcal{G}_i$ if and only if $\p\not\supseteq J_i$.
\end{proof}
Note that Corollary \ref{G_izariski} gives a description of $\mathcal{G}_i$ in terms of minors of the matrix $J_i$. This description will be used to implement the algorithm in Section \ref{section:1stAlgorithm}.
\\

In the rest of this section let $J\subset A$ be an ideal of $A$ and let $B$ be the quotient ring $A/J$. If $B$ is Cohen-Macaulay of dimension $d$ then $\bar\Omega=\Ext^{\dim A-d}_A(B,A)$ 
is a global canonical module for $B$; morover, if $B$ is generically Gorenstein then $\bar\Omega$ is isomorphic to an ideal of $B$, \cite[Prop. 3.3.18 (b)]{2}. 
In Macaulay2,  \cite{4}, we can compute explicitly a canonical ideal for B, i.e an ideal which is isomorphic to a canonical module for $B$. We start by computing a canonical module as a cokernel of a certain matrix $M$, say $B^n/V$. In order to find an ideal isomorphic to it, we look for a vector $w$ such $V$ is the kernel of $w: B^n \to B$ given by multiplication by $w$ on the left. An ad-hoc way to find such $w$ is to look among the generators of the module of syzygies of the rows of $V$.

Therefore let $B$ be generically Gorenstein and assume $\Bar\Omega\subseteq B$. Let $\Omega$ be the preimage of $\bar\Omega$ in $A$; then the following $B$-module is well defined:
\begin{equation}\label{U_e}
\mathcal{U}_{(e)}=\frac{\left(J^{[p^e]}:J\right)\cap\left(\Omega^{[p^e]}:\Omega\right)}{J^{[p^e]}}.
\end{equation}
Since $A$ is Noetherian, $\mathcal{U}_{(e)}$ is a finitely generated $A$-module (and $B$-module).\\
For every prime ideal $\p\subset A$ write $H_\p=\HH^{\dim \widehat{B}_\p}_{\p \widehat{B}_\p}  (\widehat{B}_\p)$. It follows from Theorem \ref{teo:fraction} that the $A$-module $\mathcal{F}^e(H_\p)$ consisting of the Frobenius maps on $H_\p$ is of the form:
$$
\mathcal{F}^e\left(H_\p\right)=\frac{\left(J^{[p^e]}\widehat{A}_\p:J\widehat{A}_\p\right)\cap\left(\Omega^{[p^e]}\widehat{A}_\p :\Omega \widehat{A}_\p \right)}{J^{[p^e]} \widehat{A}_\p } 
$$
and consequently $\mathcal{F}^e(H_\p)\cong\mathcal{U}_{(e)}\widehat{A}_\p$. Since $\mathcal{F}^e(H_\p)$ is generated by one element by Theorem \ref{teo:unique_generator},  $\mathcal{U}_{(e)}\widehat{A}_\p$ is principal as well.\\


From Corollary \ref{G_izariski} with $M=\mathcal{U}_{(e)}$ it follows that for every prime ideal $\p\in \Spec(A)=\bigcup_i \mathcal{G}_i$ there exists an $i$ such that $\p\in \mathcal{G}_i$ and the $A$-module $\mathcal{U}_{(e)}\widehat{A}_\p$ is generated by one element which is precisely the image of $g_i$. 

With the notation above, we prove our main result.
\begin{teo}\label{main_thm}
For every $e$, the set $\mathcal{B}_e= \left\{\p\in \Spec(A)\vert \HSL (H_\p)<e \right\}$ is Zariski open.
\end{teo}
\begin{proof}
Let $u_1,\cdots, u_s$ be a set of generators for $\mathcal{U}_{(e)}$ and write
$$\mathcal{G}_i= \{\p\in\Spec(A) \ \vert \mathcal{U}_{(e)}\widehat{A}_\p \text{\emph{ is generated by} } \frac{u_i}{1} \}.$$
Define
$$\Omega_{i,e}=\frac{I_e\left(u_i^{\nu_e}(\Omega)\right)}{I_{e+1}(u_i^{\nu_{e+1}}(\Omega))}$$
then it follows from Proposition \ref{I_eCommutes} that 
$$\frac{I_e(\bar{u}_i^{\nu_e}(\widehat{\Omega}_\p))}{I_{e+1}(\bar{u}_i^{\nu_{e+1}}(\widehat{\Omega}_\p))}=(\widehat\Omega_{i,e})_\p$$
for every prime ideal $\p$. \\Note that for every $i=1,\cdots,s$ the set 
$\Supp(\Omega_{i,e})^C=\left\{\p \ \vert \ \left(\widehat\Omega_{i,e}\right)_\p=0\right\} $
is open.\\
If $\p$ is such that $\HSL(H_\p)< e$ then $\p\in \mathcal{G}_i$ for some $i$; we can then use $\bar{u}_i$ to compute $\Omega_{i,e}$ and $(\widehat\Omega_{i,e})_\p=0$ i.e. $\p\in \Supp(\Omega_{i,e})^C $; therefore 
$$\p\in \bigcup_i \left(\Supp(\Omega_{i,e})^C \cap \mathcal{G}_i\right). $$
Viceversa, let $\p\in \bigcup_i \left(\Supp(\Omega_{i,e})^C \cap \mathcal{G}_i\right)$ then $\p\in \Supp(\Omega_{j,e})^C\cap\mathcal{G}_j$ for some $j$. Compute $\HSL(H_\p)$ using $u_j$. Since $\p\in \Supp(\Omega_{j,e})^C$ then $\left(\widehat\Omega_{j,e}\right)_\p=0$ and so $\HSL\left(H_\p\right)< e$.
In conclusion
$$\left\{ \p\in \Spec(A)\vert \HSL \left(H_\p\right)<e \right\}=\bigcup_i \left(\Supp\left(\Omega_{i,e}\right)^C \cap \mathcal{G}_i \right)$$
and therefore $\mathcal{B}_e$ is Zariski open.  
\end{proof}

\begin{corol}
The index of nilpotency is bounded.
\end{corol}
For $e=1$ we have the following.
\begin{corol}\label{locus_open}
The $F$-injective locus of the top local cohomology of a quotient of a polynomial ring is open.
\end{corol}

\section{The Computation of the $\HSL$ Loci}\label{section:1stAlgorithm}
In the case where a ring $S=R/I$ of positive characteristic $p$ is a Cohen-Macaulay domain we have an explicit algorithm to compute the F-injective locus $\mathcal{B}_e$ of $S$  
for every positive integer $e$. \\
Using the same notation as in Theorem \ref{main_thm} we have that $\mathcal{B}_e=\bigcup_{i=1}^s (\Supp(\Omega_{i,e})^C \cap \mathcal{G}_i)$. Because $\Supp(\Omega_{i,e})=V(\Ann_R(\Omega_{i,e}))$  and the sets $\mathcal G_i$ are of the form $V(K_i)^c$ for some ideals $K_1,\cdots,K_s\subset R$, we can then write $\mathcal{B}_e$  as
\begin{equation*}
\begin{split}
&\bigcup_i V(\Ann_R(\Omega_{i,e}))^c\cap V(K_i)^c=\\
&\bigcup_i \left(V(\Ann_R(\Omega_{i,e}))\cup V(K_i)\right)^c=\\
&\bigcup_i V(\Ann_R(\Omega_{i,e}) K_i )^c=\\
&\left( \bigcap_i V(\Ann_R(\Omega_{i,e})K_i)\right)^c=\\
&\left[V\left(\sum_i \Ann_R(\Omega_{i,e})K_i\right)\right]^c.
\end{split}
\end{equation*}
Therefore, given a positive integer $e$ and a Cohen-Macaulay domain $S$, an algorithm to find the locus $\mathcal{B}_e$ can be described as follows.
\begin{enumerate}
\item Compute a canonical module for $S$, then find an ideal $\Omega\subseteq S$ which is isomorphic to it.
\item Find the $R$-module of the Frobenius maps on $\Hs$ defined in (\ref{U_e}) as $$\U_{(e)}=\frac{\left(I^{[p^e]}:I\right)\cap\left(\Omega^{[p^e]}:\Omega\right)}{I^{[p^e]}}$$
as the cokernel of a matrix $X\in \Mat$. 
\item Find the generators $u_1,\cdots,u_s$ of $\U_{(e)}$.
\item Compute the ideals $K_i$'s of the $(s-1)\times(s-1)$-minors of $X$.

\item  For every generator $u_i$, compute the ideal $\Omega_{i,e}=I_e(u_i^{\nu_e}\Omega )/I_{e+1}(u_i^{\nu_{e+1}}\Omega )$. 
\item Compute $\mathcal{B}_e$ as $\left[V\left(\sum_i \Ann_R(\Omega_{i,e})K_i\right)\right]^c$.
\end{enumerate}
We now make use of the algorithm above to compute the loci in an example. The algorithm has been implemented in Macaulay2.
\begin{esem}
Let $R$ be the polynomial ring $\Z_2[x_1,\cdots,x_5]$ and let $I$ be the ideal $I=(x_2^2+x_1x_3,x_1x_2x_4^2+x_3^3x_5,x_1^2x_4^2+x_2x_3^2x_5)$. The quotient ring $S=R/I$ is a domain because $I$ is prime and it is Cohen-Macaulay of type 2 so it is not Gorenstein. A canonical module for $S$ is given by $\Ext^{\dim R -\dim S}(S,R)$ and can be produced as the cokernel of the matrix 
\[ \left( \begin{array}{ccccccccccccccc}
x_2   & x_1  & x_3^2x_5 \\    
x_3   & x_2  & x_1x_4^2 
\end{array} \right)\] 
which is isomorphic to an ideal $\Omega$ which is the image in $S$ of the ideal  $(x_2,x_1,x_2^2+x_1x_3,x_1x_2x_4^2+x_3^3x_5,x_1^2x_4^2+x_2x_3^2x_5)$ in $R$.\\
The $R$-module $\U_{(e)}$ of the Frobenius maps on $\Hs$ turns out to be given by the cokernel of the one-row matrix 
\[ X=\left( \begin{array}{ccccccccccccccc}
x_2^2+x_1x_3   & x_1x_2x_4^2+x_3^3x_5  & x_1^2x_4^2+x_2x_3^2x_5 
\end{array} \right)\] 
whose generator is $u=x_1^2x_2^2x_4^2+x_1^3x_3x_4^2+x_2^3x_3^2x_5+x_1x_2x_3^3x_5$.
Since $X$ has only one row then the computation of $\mathcal{B}_e$ reduces to $\mathcal{B}_e=\big[V(\Ann_R(\Omega_{e})\big]^c$. It turns out that 
$I_1(u^{\nu_1}\Omega)=(x_1x_4,x_2x_3,x_1x_3,x_3^3x_5,x_2^2+x_1x_3,x_1x_2x_4^2+x_3^3x_5,
x_1^2x_4^2+ x_2x_3^2x_5)$, $I_2(u^{\nu_2}\Omega)=(x_1x_4,x_2x_3,x_1x_2x_4,x_1^2x_4,x_3^3x_5,x_2^2+x_1x_3,
x_1x_2x_4^2+x_3^3x_5,x_1^2x_4^2+ x_2x_3^2x_5)=I_3(u^{\nu_3}\Omega)$. Consequently, being $\Omega_{0}=\frac{\Omega}{I_{1}(u^{\nu_{1}}(\Omega))}$ and
$\Omega_{1}=\frac{I_1(u^{\nu_1}(\Omega))}{I_{2}(u^{\nu_{2}}(\Omega))}$, we have $\mathcal{B}_0=V(x_3,x_2x_4,x_1x_4,x_2^2)^c = V(x_1,x_2,x_3)^c\cup V(x_2,x_3,x_4)^c$, 
$\mathcal{B}_1=V(x_1,x_2,x_3)^c$ and 
$\mathcal{B}_e=V(1)^c$ for every $e>1$. 
In other words, the $\HSL$-number can be at the most 2. More precisely, if we localise $S$ at a prime that does not contain the prime ideal $(x_3,x_2x_4,x_1x_4,x_2^2)$ then we get an $F$-injective module. Outside $(x_1,x_2,x_3)$ the $\HSL$-number is less or equal to 1; On $V(x_1,x_2,x_3)$ the $\HSL$ number is exactly 2.
\end{esem}


\section{Test Exponents for Frobenius Closures and $\HSL$ Numbers} \label{Section:application}

Let $S$ be a ring of characteristic $p$ and $J\subseteq S$ an ideal.
\begin{defi}
The \emph{Frobenius closure} of $J$ is the ideal
$$J^F=\left\{a\in S \ \vert \  a^{p^e}\in J^{[p^e]} \text{ for some } e>0  \right\}.$$
\end{defi}
Note that if $a^{p^{\bar e}}\in J^{[p^{\bar e}]} $ then $a^{p^{e}}\in J^{[p^{e}]} $ for every $e>\bar e. \\
$Let $g_1,\cdots g_n$ be a set of generators for $J^F$. For each generator $g_i$ let $e_i$ be the integer such that $g_i^{p^{e_i}}\in J^{p^{e_i}}$. If we then choose $\bar e=  \max \{e_1,\cdots,e_n\}$ then $(J^F)^{[p^{\bar e}]}\subseteq J^{[p^{\bar e}]}$. We say that $\bar e$ is a \emph{test exponent for the Frobenius closure of} $J$.\\
With the notation introduced in Section \ref{section:introduction}, we have the following.
\begin{teo}\cite[Theorem 2.5]{8}\label{testExp}
Let $(S,\m)$ be a local, Cohen-Macaulay ring and let $\underbar x=x_1,\cdots,x_d$ be a system of parameters.  Then the test exponent for the ideal $(\underbar x)$ is $\bar e= \HSL(\Hs)$.   
\end{teo}

\begin{proof}
$$\Hs=\lim_{\overset{\longrightarrow}{t}} \left( \frac{S}{\underbar x}
\overset{x_1\cdots x_n}{\longrightarrow} \cdots \overset{x_1\cdots x_n}{\longrightarrow} \
\frac{S}{\underbar{x}^t} \overset{x_1\cdots x_n}{\longrightarrow} \cdots \right)$$
has a natural Frobenius action $T$ which we can define on a generic element of the direct limit as 
$$T[a+\underbar x^t]=a^p+\underbar x^{pt}.$$ Therefore $a^{p^e} \in \underbar x^{p^e}$ if and only if $T^e[a+\underbar x^t]=0$ i.e. $[a+\underbar x^t]$ is nilpotent and we can take $\bar e=\HSL(\Hs)$.
\end{proof}
\begin{corol}\label{testexponent}
Let $S$ be the quotient of a polynomial ring and let $\epsilon$ be the bound for $\{\HSL(\HH^{\text{dim} S}_\m(S)) \ \vert \ \m \text{ is maximal}\}$. If $J\subseteq S$ is locally a parameter ideal (i.e. for every maximal ideal  $\m\supseteq J$, $J_\m$ is a parameter ideal) then $(J^F)^{[p^\epsilon]}=J^{p^\epsilon}$.   
\end{corol}



\section*{Acknowledgements}
The author is very grateful to her supervisor Moty Katzman for the helpful discussions about the problem.


\end{document}